\documentclass[12pt,twoside,a4paper]{amsart}
\usepackage{amssymb}
\usepackage{graphicx}
\usepackage[all]{xy}

\date{\today}

\usepackage[latin1]{inputenc}
\usepackage[T1]{fontenc}

\def\dbar{\bar\partial}

\def\R{{\mathbb R}}
\def\C{{\mathbb C}}
\def\P{{\mathbb P}}

\def\O{{\mathcal O}}

\def\L{{\mathcal L}}

\def\Re{{\rm Re\,  }}

\def\L{{\mathcal L}}
\def\U{{\mathcal U}}

\def\codim{\text{codim}\,}
\def\ann{\text{ann}\,}

\def\1{\mathbf 1}

\def\m{{\mathfrak m}}

\def\np{{\rext NP}}
\def\soc{\text{Soc}}

\def\a{{\mathfrak a}}
\def\ord{\text {ord}}
\def\vE{\text {ord}_E}
\def\jac{\text {Jac}}
\def\I{{\mathcal I}}
\def\np{\text{NP}}
\def\1{\mathbf 1}

\newcommand{\fa}{\mathfrak{a}}

\newcommand{\fm}{\mathfrak{m}}

\newcommand{\cI}{\mathcal{I}}

\newcommand{\cO}{\mathcal{O}}

\def\be{\begin{equation}}
\def\ee{\end{equation}}

\newtheorem{thm}{Theorem}[section]
\newtheorem{lma}[thm]{Lemma}
\newtheorem{cor}[thm]{Corollary}
\newtheorem{prop}[thm]{Proposition}

\newtheorem*{thmA}{Theorem A}
\newtheorem*{thmB}{Theorem B}
\newtheorem*{thmC}{Theorem C}
\newtheorem*{questionD}{Question D}

\theoremstyle{remark}
\newtheorem{remark}[thm]{Remark}
\newtheorem{ex}[thm]{Example}

\theoremstyle{definition}
\newtheorem{df}[thm]{Definition}

\theoremstyle{remark}

\numberwithin{equation}{section}

\begin{document}

\title[On Bochner-Martinelli residue currents...]{On Bochner-Martinelli residue currents and their annihilator ideals}

\date{\today}
\thanks{First author partially supported by the NSF
  and the Swedish Research Council.
  Second author partially supported by 
  the Royal Swedish Academy of Sciences and the Swedish Research Council.}
\author{Mattias Jonsson \& Elizabeth Wulcan}

\address{Dept of Mathematics, University of Michigan, Ann Arbor \\ MI 48109-1109\\ USA}

\email{mattiasj@umich.edu, wulcan@umich.edu}

\subjclass{}

\keywords{}

\begin{abstract}
We study the residue current $R^f$ of Bochner-Martinelli type
associated with a tuple $f=(f_1,\dots,f_m)$
of holomorphic germs at $0\in\mathbf{C}^n$,
whose common zero set equals the origin.
Our main results are a geometric description of $R^f$
in terms of the Rees valuations associated with the ideal $(f)$ generated by $f$ and a characterization of when the annihilator ideal of $R^f$ equals $(f)$.
\end{abstract}

\maketitle

\section{Introduction}\label{intro}
Residue currents are generalizations of classical one-variable residues and can be thought of as currents representing ideals of holomorphic functions. In ~\cite{PTY} Passare-Tsikh-Yger introduced residue currents based on the Bochner-Martinelli kernel. Let $f=(f_1,\ldots,f_m)$ be a tuple of (germs of) holomorphic functions at $0\in\C^n$, such that $V(f)=\{f_1=\ldots = f_m=0\}=\{0\}$. (Note that we allow $m > n$.) For each ordered multi-index $\I=\{i_1,\ldots,i_n\}\subseteq \{1,\ldots, m\}$ let 
\begin{equation}\label{blanc}
R^f_\I=\dbar |f|^{2\lambda}\wedge
c_n \sum_{\ell=1}^n(-1)^{\ell-1}
\frac
{\overline{f_{i_\ell}}\bigwedge_{q\neq \ell} \overline{df_{i_q}}}
{|f|^{2n}}\bigg |_{\lambda=0},
\end{equation}
where $c_n=(-1)^{n(n-1)/2}(n-1)!$, 
$|f|^2=|f_1|^2+\ldots+|f_m|^2$, and $\alpha |_{\lambda =0}$ denotes the analytic continuation of the form $\alpha$ to $\lambda=0$. 
Moreover, let $R^f$ denote the vector-valued current with entries $R^f_\I$; we will refer to this as the \emph{Bochner-Martinelli residue current} associated with $f$. Then $R^f$ is a well-defined $(0,n)$-current with support at the origin and $\overline g R^f_\I=0$ if $g$ is a holomorphic function that vanishes at the origin. It follows that the coefficients of the $R^f_\I$ are just finite sums of holomorphic derivatives at the origin. 

Let $\O^n_0$ denote the local ring of germs of holomorphic functions at $0\in\C^n$. Given a current $T$ let $\ann T$ denote the (holomorphic) \emph{annihilator ideal} of $T$, that is, 
\[
\ann T=\{h\in \O^n_0, hT=0\}.
\]
Our main result concerns $\ann R^f=\bigcap \ann R^f_\I$. Let $(f)$ denote the ideal generated by the $f_i$ in $\O^n_0$. Recall that $h\in\O^n_0$ is in the integral closure of $(f)$, denoted by $\overline{(f)}$, if $|h|\leq C|f|$, for some constant $C$. Moreover, recall that $(f)$ is a \emph{complete intersection ideal} if it can be generated by $n=\codim V(f)$ functions. Note that this condition is slightly weaker than $\codim V(f)=n=m$.

\begin{thmA}
Suppose that $f$ is a tuple of germs of holomorphic functions at $0\in\C^n$ such that $V(f)=\{0\}$. Let $R^f$ be the corresponding Bochner-Martinelli residue current. Then
\begin{equation}\label{eqa}
\overline{(f)^n}\subseteq\ann R^f\subseteq (f).
\end{equation}
The left inclusion in \eqref{eqa} is strict whenever $n\geq 2$. The right inclusion is an equality if and only if $(f)$ is a complete intersection ideal. 
\end{thmA}

The new results in Theorem A are the last two statements. The left and right inclusions in \eqref{eqa} are due to Passare-Tsikh-Yger ~\cite{PTY} and Andersson ~\cite{A}, respectively. Passare-Tsikh-Yger defined currents $R^f_\I$ also when $\codim V(f) < n$. The inclusions \eqref{eqa} hold true also in this case; one even has 
$\overline{(f)^{\min (m,n)}}\subseteq\ann R^f\subseteq (f)$. 
Furthermore, Passare-Tsikh-Yger showed that $\ann R^f=(f)$ if $m=\codim V(f)$. More precisely, they proved that in this case the only entry $R^f_{\{1,\ldots,m\}}$ of $R^f$ coincides with the classical \emph{Coleff-Herrera product} 
\begin{equation*}
R^f_{CH}=
\dbar \left [\frac{1}{f_1}\right ]\wedge\cdots\wedge 
\dbar \left [\frac{1}{f_m}\right ],
\end{equation*}
introduced in ~\cite{CH}. 
The current $R^f_{CH}$ represents the ideal in the sense that $\ann R^f_{CH}=(f)$ as proved by Dickenstein-Sessa ~\cite{DS} and Passare ~\cite{P}. This so-called \emph{Duality Principle} has been used for various purposes, see ~\cite{BGVY}. 
Any ideal of holomorphic functions can be represented as the annihilator ideal of a (vector valued) residue current. However, in general this current is not as explicit as the Coleff-Herrera product, see ~\cite{AW}. 

Thanks to their explicitness Bochner-Martinelli residue currents have found many applications, see for example ~\cite{AG}, ~\cite{ASS}, ~\cite{BY}, and ~\cite{VY}. Even though the right inclusion in \eqref{eqa} is strict in general, $\ann R^f$ is large enough to in some sense capture the ``size'' of $(f)$. For example \eqref{eqa} (or rather the general version stated above) gives a proof of the Brian{\c c}on-Skoda Theorem ~\cite{BS}, see also ~\cite{A}. The inclusions in \eqref{eqa} are central also for the applications mentioned above. 

The proof of Theorem A has three ingredients.
First, we use a result of Hickel ~\cite{Hi} relating the ideal $(f)$ to
the Jacobian determinant of $f$.
Second, we rely on a result by Andersson, which says that
under suitable hypotheses, the current he constructs in ~\cite{A} is
independent of the choice of Hermitian metric, see also Section ~\ref{rescurr}.

The third ingredient, which is of independent interest,
is a geometric description of the Bochner-Martinelli
current, and goes as follows.
Let $\pi:X\to (\C^n,0)$ be a \emph{log-resolution} of $(f)$, see Definition ~\ref{loglog}. We say that a multi-index $\I=\{i_1,\ldots,i_n\}$ is \emph{essential} if there is an exceptional prime $E\subseteq\pi^{-1}(0)$ of $X$ such that the mapping $[f_{i_1}\circ \pi:\ldots :f_{i_n}\circ\pi]: E\to \C\P^{n-1}$ is surjective and moreover $\vE(f_{i_k})\leq \vE(f_{\ell})$ for $1\leq k\leq n, 1\leq\ell\leq m$, see Section ~\ref{essential} for more details. The valuations $\vE$ are precisely the \emph{Rees valuations} of $(f)$. 

\begin{thmB}
Suppose that $f$ is a tuple of germs of holomorphic functions at $0\in\C^n$ such that $V(f)=\{0\}$. Then the current $R^f_\I\not\equiv 0$ if and only if $\I$ is essential.
\end{thmB}

As is well known, one can view $R^f$ as the pushforward of
a current on a log-resolution of $(f)$.
The support on  the latter current is then exactly the exceptional
components associated with the Rees valuations of $(f)$, see Section ~\ref{coeff}.

Recall that if $(f)$ is a complete intersection ideal, then $(f)$ is in fact generated by $n$ of the $f_i$. This follows for example by Nakayama's Lemma.

\begin{thmC}
Suppose that $f$ is a tuple of germs of holomorphic functions at $0\in\C^n$ such that $V(f)=\{0\}$ and such that $(f)$ is a complete intersection ideal. Then $\I=\{i_1,\ldots,i_n\}$ is essential if and only if $f_{i_1}, \ldots f_{i_n}$ generates $(f)$. Moreover 
\begin{equation}\label{formen}
R^f_\I= C_\I ~\dbar \left [\frac{1}{f_{i_1}}\right ]\wedge\cdots\wedge 
\dbar \left [\frac{1}{f_{i_n}}\right ],
\end{equation}
where $C_\I$ is a non-zero constant.
\end{thmC}
Theorems B and C generalize previous results for monomial ideals. In ~\cite{W} an explicit description of $R^f$ is given in case the $f_i$ are monomials; it is expressed in terms of the Newton polytope of $(f)$. From this description a monomial version of Theorem A can be read off. Also, it follows that in the monomial case $\ann R^f$ only depends on the ideal $(f)$ and not on the particular generators $f$. This motivates the following question.

\begin{questionD}
Let $f$ be a tuple of germs of holomorphic functions such that $V(f)=\{0\}$. Let $R^f$ be the corresponding Bochner-Martinelli residue current. Is it true that $\ann R^f$ only depends on the ideal $(f)$ and not on the particular generators $f$? 
\end{questionD}
Computations suggest that the answer to Question D may be positive; see Remark ~\ref{Dremark}. If $\codim V(f)<n$, then $\ann R^f$ may in fact depend on $f$ even though the examples in which this happens are somewhat pathological, see for example ~\cite[Example~3]{A}. A positive answer to Question D would imply that we have an ideal canonically associated with a given ideal; it would be interesting to understand this new ideal algebraically.

This paper is organized as follows. In Sections ~\ref{rescurr} and ~\ref{reesprelim} we present some necessary background on residue currents and Rees valuations, respectively. The proof of Theorem B occupies Section ~\ref{coeff}, whereas Theorems A and C are proved in Section ~\ref{annsection}. In Section ~\ref{decompsection} we discuss a decomposition of $R^f$ with respect to the Rees valuations of $(f)$. In the last two sections we interpret our results in the monomial case and illustrate them by some examples.

\smallskip

\noindent
\textbf{Acknowledgment:}
We would like to thank Mats Boij and H{\aa}kan Samuelsson for valuable discussions. This work was partially carried out when the authors were visiting
the Mittag-Leffler Institute.

\section{Residue currents}\label{rescurr}

We will work in the framework from Andersson ~\cite{A} and use the fact that the residue currents $R^f_\I$ defined by \eqref{blanc} appear as the coefficients of a vector bundle-valued current introduced there. Let $f=(f_1,\ldots, f_m)$ be a tuple of germs of holomorphic functions at $0\in\C^n$. We identify $f$ with a section of the dual bundle $V^*$ of a trivial vector bundle $V$ over $\mathbb C^n$ of rank $m$, endowed with the trivial metric. 
If $\{e_i\}_{i=1}^m$ is a global holomorphic frame for $V$ and $\{e^*_i\}_{i=1}^m$
is the dual frame, we can write $f=\sum_{i=1}^m f_i e_i^*$. We let $s$ be the dual section $s=\sum_{i=1}^m \bar f_i e_i$ of $f$.

Next, we let 
\begin{equation*}
u =
\sum_\ell\frac{s\wedge(\dbar s)^{\ell-1}}{|f|^{2\ell}},
\end{equation*}
where $|f|^2=|f_1|^2+\ldots + |f_m|^2$. Then $u$ is a section of $\Lambda(V\oplus T_{0,1}^*(\mathbb C^n))$ (where $e_j\wedge d\bar z_i=-d\bar z_i\wedge e_j$), that is clearly well defined and smooth outside $V(f)=\{f_1=\ldots = f_m=0\}$, and moreover 
\begin{equation*}
\dbar|f|^{2\lambda}\wedge u, 
\end{equation*}
has an analytic continuation as a current to $\Re \lambda > -\epsilon$. We denote the value at $\lambda=0$ by $R$. Then $R$ has support on $V(f)$ and $R=R_p+\ldots +R_\mu$, where $p=\codim V(f)$, $\mu=\min(m,n)$, and where $R_k\in \mathcal D'_{0,k}(\mathbb C^n,\Lambda^k V)$. In particular if $V(f)=\{0\}$, then 
$R=R_n$. 

We should remark that Andersson's construction of residue currents works for sections of any holomorphic vector bundle equipped with a Hermitian metric. 
In our case (trivial bundle and trivial metric), however, the coefficients of $R$ are just the residue currents $R^f_\I$ defined by Passare-Tsikh-Yger ~\cite{PTY}. Indeed, for $\I=\{i_1,\ldots, i_k\}\subseteq\{1,\ldots, m\}$ let $s_\I$ be the section $\sum_{j=1}^k \bar f_{i_j} e_{i_j}$, that is, the dual section of $f_\I=\sum_{j=1}^k f_{i_j} e^*_{i_j}$. Then we can write 
$u$ as a sum, taken over subsets $\I=\{i_1,\ldots, i_k\}\subseteq\{1,\ldots, m\}$, of terms
\begin{equation*}
u_\I= 
\frac{s_\I\wedge (\dbar s_\I)^{k-1}}{|f|^{2k}}.
\end{equation*}
The corresponding current,
\begin{equation*}
\dbar|f|^{2\lambda}\wedge u_\I|_{\lambda=0}
\end{equation*}
is then merely the current 
\begin{equation*}
R^f_\I:=\dbar |f|^{2\lambda}\wedge
c_k \sum_{\ell=1}^k(-1)^{\ell-1}
\frac
{\overline{f_{i_\ell}}\bigwedge_{q\neq \ell} \overline{df_{i_q}}}
{|f|^{2k}}\bigg |_{\lambda=0},
\end{equation*}
where $c_k=(-1)^{k(k-1)/2}(k-1)!$, times the frame element $e_\I=e_{i_k}\wedge\cdots\wedge e_{i_1}$; we denote it by $R_\I$. Throughout this paper we will use the notation $R^f$ for the vector valued current with entries $R^f_\I$, whereas $R$ and $R_\I$ (without the superscript $f$), respectively, denote the corresponding $\Lambda^n V$-valued currents.

Let us make an observation that will be of further use. If the section $s$ can be written as $\mu s'$ for some smooth function $\mu$ we have the following homogeneity:
\begin{equation}\label{homogen}
s\wedge (\dbar s)^{k-1}=\mu^k s'\wedge (\dbar s')^{k-1},
\end{equation}
that holds since $s$ is of odd degree.

Given a holomorphic function $g$ we will use the notation $\dbar [1/g]$ for the value at $\lambda=0$ of $\dbar |g|^{2\lambda}/g$ and analogously by $[1/g]$ we will mean $|g|^{2\lambda}/g|_{\lambda=0}$, that is, the principal value of $1/g$. 
We will use the fact that 
\begin{equation}\label{envar}
v^\lambda |\sigma|^{2\lambda}\frac{1}{\sigma^a}\bigg |_{\lambda=0} = \left [ \frac{1}{\sigma^a} \right ]
\quad \text{ and } \quad 
\dbar(v^\lambda |\sigma|^{2\lambda})\frac{1}{\sigma^a}\bigg |_{\lambda=0} = \dbar \left [ \frac{1}{\sigma^a} \right ],
\end{equation}
if $v=v(\sigma)$ is a strictly positive smooth function; compare to ~\cite[Lemma~2.1]{A}.

\subsection{Restrictions of currents and the Standard Extension Property}\label{pseudo}
In ~\cite{AW2} the class of \emph{pseudomeromorphic} currents is introduced. The definition is modeled on the residue currents that appear in various works such as ~\cite{A} and ~\cite{PTY}; a current is pseudomeromorphic if it can be written as a locally finite sum of push-forwards under holomorphic modifications of currents of the simple form 
\[[1/(\sigma_{q+1}^{a_{q+1}}\cdots \sigma_{n}^{a_{n}})]\dbar[1/\sigma_1^{a_1}]\wedge\cdots\wedge \dbar[1/\sigma_q^{a_q}]
\wedge \alpha, \]
where $\sigma_j$ are some local coordinates and $\alpha$ is a smooth form. In particular, all currents that appear in this paper are pseudomeromorphic. 

An important property of pseudomeromorphic currents is that they can be restricted to varieties and, more generally, constructible sets. More precisely, they allow for multiplication by characteristic functions of constructible sets so that ordinary calculus rules holds. In particular, 
\begin{equation}\label{mjau}
\1_V (\beta \wedge T)= \beta \wedge (\1_V T),
\end{equation}
if $\beta$ is a smooth form. 
Moreover, suppose that $S$ is a pseudomeromorphic current on a manifold $Y$, that $\pi:Y\to X$ is a holomorphic modification, and that $A\subseteq Y$ is a constructible set. Then
\begin{equation}\label{joy}
\1_A(\pi_*S)=\pi_*(\1_{\pi^{-1}(A)}S).
\end{equation}

A current $T$ with support on an analytic variety $V$ (of pure dimension) 
is said to have the so-called \emph{Standard Extension Property (SEP) with respect to $V$} if it is equal to its standard extension in the sense of ~\cite{Bj}; this basically means that it has no mass concentrated to sub-varieties of $V$. 
If $T$ is pseudomeromorphic, $T$ has the SEP with respect to $V$ if and only if 
$\1_WT=0$ for all subvarieties $W\subset V$ of smaller dimension than $V$, see ~\cite{A6}.  
We will use that the current $\dbar[1/\sigma_i^a]$ has the SEP with respect to $\{\sigma_i=0\}$; in particular, $\dbar[1/\sigma_i^a]\1_{\{\sigma_j=0\}}=0$. 
If $S$ and $\pi$ are as above and we moreover assume that $S$ has the SEP with respect to an analytic variety $W$, then $\pi_* S$ has the SEP with respect to $\pi^{-1}(W)$.

\section{Rees valuations}\label{reesprelim}
%
%
\subsection{The normalized blowup and Rees valuations}
We will work in a local situation. Let $\O_0^n$ denote the local ring of germs of holomorphic functions at $0\in\C^n$, and let $\m$ denote its maximal ideal. Recall that an ideal $\a\subset\O_0^n$ is \emph{$\m$-primary} if its associated zero locus $V(\fa)$ is equal to the origin.

Let $\fa\subset\O_0^n$ be an $\fm$-primary ideal. 
The \emph{Rees valuations} of $\fa$ are defined
in terms of the normalized blowup $\pi_0:X_0\to(\C^n,0)$
of $\fa$. 
Since $\fa$ is $\fm$-primary, 
$\pi_0$
is an isomorphism outside $0\in\C^n$ and 
$\pi_0^{-1}(0)$ is the union of finitely
many prime divisors $E\subset X_0$. The Rees valuations  
of $\fa$ are then the associated (divisorial) valuations
$\ord_E$ on $\cO_0^n$:
$\ord_E(g)$ is the order of vanishing of $g$
along $E$.

The blowup of an ideal
is defined quite generally 
in~\cite[Ch.II, \S7]{Ha}.
We shall make use of the following
more concrete description, see ~\cite[p.~332]{T}.
Let $f_1,\dots,f_m$ be generators of $\fa$ 
and consider the rational map
$\psi:(\C^n,0)\dashrightarrow\P^{m-1}$ given by
$\psi=[f_1:\dots:f_m]$. 
Then $X_0$ is the normalization of the closure of the graph 
of $\psi$, and $\pi_0:X_0\to(\C^n,0)$ is the natural projection.
Denote by $\Psi_0:X_0\to\P^{m-1}$ the other projection.
It is a holomorphic map. The image under $\Psi_0$
of any prime divisor $E\subset\pi_0^{-1}(0)$ has dimension $n-1$.

%
%
\subsection{Log resolutions}\label{logres}
The normalized
blowup can be quite singular, making it difficult to use for
analysis. Therefore, we shall use a 
\emph{log-resolution} of $\fa$, 
see ~\cite[Definition~9.1.12]{L2}.
\begin{df}\label{loglog}
  A \emph{log-resolution} of $\fa$ is a holomorphic modification  
  $\pi:X\to(\C^n,0)$, where $X$ is a complex manifold, such that
  \begin{itemize}
  \item
    $\pi$ is an isomorphism above $\C^n\setminus\{0\}$:
  \item
    $\fa\cdot\cO_X=\cO_X(-Z)$, where $Z=Z(\fa)$ 
    is an effective divisor on $X$
    with simple normal crossings support.
  \end{itemize}
\end{df}
The simple normal crossings condition means that 
the exceptional divisor $\pi^{-1}(0)$ is a
union of finitely many prime divisors $E_1,\dots,E_N$,
called \emph{exceptional primes},
and at any point $x\in\pi^{-1}(0)$ we can pick local coordinates
$(\sigma_1,\dots,\sigma_n)$ at $x$ such that 
$\pi^{-1}(0)=\{\sigma_1\cdot\dots\cdot\sigma_p=0\}$
and for each exceptional prime $E$, either 
$x\not\in E$, or $E=\{\sigma_i=0\}$ 
for some $i\in\{1,\dots p\}$.

If we write $Z=\sum_{j=1}^Na_jE_j$, then the
condition $\fa\cdot\cO_X=\cO_X(-Z)$
means that (the pullback to $X$ of)
any holomorphic germ $g\in\a$ vanishes to order 
at least $a_j$ along each $E_j$. 
Moreover, in the notation above, if $x\in\pi^{-1}(0)$
and $E_{j_k}=\{\sigma_k=0\}$, $1\le k\le p$ are the 
exceptional primes containing $x$, then there exists
$g\in\fa$ such that $g=\sigma_1^{a_1}\dots \sigma_p^{a_p}u$, where 
$u$ is a unit in $\cO_{X,x}$, that is, $u(x)\neq 0$. 

The existence of a log-resolution is a consequence of 
Hironaka's theorem on resolution of singularities.
Indeed, the ideal $\fa$ is already principal on the 
normalized blowup $X_0$, 
so it suffices to pick $X$ as a desingularization of $X_0$. 
This gives rise to a commutative diagram
\begin{equation*}
  \xymatrix{
    &&
    X
    \ar[d]_{\varpi}
    \ar[ddll]_{\pi}
    \ar[ddrr]^{\Psi}
    &&
    \\
    &&
    X_0
    \ar[dll]_(.3){\pi_0}
    \ar[drr]^(.3){\Psi_0}
    &&
    \\
    (\C^n,0)
    \ar@{-->}[rrrr]^{\psi}
    &&&&
    \P^{m-1}
  }
\end{equation*}
Here $\Psi:X\to\P^{m-1}$ is holomorphic.

Every exceptional prime $E$ of a log resolution 
$\pi:X\to(\C^n,0)$ of
$\fa$ defines a divisorial valuation $\ord_E$, but not all
of these are Rees valuations of $\fa$. 
If $\ord_E$ is a Rees valuation, we call $E$ a
\emph{Rees divisor}.
From the diagram above we see:
\begin{lma}
  An exceptional prime $E$ of $\pi$ is a
  Rees divisor of $\fa$ if and only if its 
  image $\Psi(E)\subset\P^{m-1}$ 
  has dimension $n-1$.
\end{lma}
For completeness we give two results, the second of which
will be used in Example~\ref{icke-monom}.
\begin{prop}\label{intrees}
  Let $E$ be an exceptional prime of a log resolution 
  $\pi:X\to(\C^n,0)$ of $\fa$.
  Then the intersection number $((-Z(\fa))^{n-1}\cdot E)$ is
  strictly positive if $E$ is a Rees divisor of $\fa$ 
  and zero otherwise.
\end{prop}
\begin{proof}
  On the normalized blowup $X_0$, we may write 
  $\fa\cdot\cO_{X_0}=\cO_{X_0}(-Z_0)$,
  where $-Z_0$ is an \emph{ample} divisor.
  Then $\fa\cdot\cO_X=\cO_X(-Z)$, where
  $Z=\varpi^*Z_0$. It follows that 
  $((-Z^{n-1})\cdot E)=((-Z_0^{n-1})\cdot\varpi_*E)$.
  The result follows since $-Z_0$ is ample and 
  since $E$ is a Rees divisor if and only if
  $\varpi_*(E)\ne0$.
\end{proof}
\begin{cor}\label{dimension2}
  In dimension $n=2$, the Rees valuations of a product
  $\fa=\fa_1\cdot\dots\cdot\fa_k$ of $\fm$-primary 
  ideals is the union of the Rees valuations of the $\fa_i$.
\end{cor}
\begin{proof}
  Pick a common log-resolution $\pi:X\to(\C^n,0)$ of
  all the $\fa_i$. 
  Then $\fa_i\cdot\cO_X=\O_X(-Z_i)$
  and $\fa\cdot\cO_X=\O_X(-Z)$, 
  where $Z=\sum_iZ_i$. 
  Fix an exceptional prime $E$. 
  By Proposition~\ref{intrees} we have 
  $(Z_i\cdot E)\le 0$ with strict inequality if and only
  if $E$ is a Rees divisor of $\fa_i$.
  Thus $(Z\cdot E)=\sum_i(Z_i\cdot E)\le0$ 
  with strict inequality if and only $E$ is a
  Rees divisor of some $\fa_i$.
  The result now follows from Proposition~\ref{intrees}.
\end{proof}
%
%
\subsection{Essential multi-indices}\label{essential}
In our situation, we are given an $\fm$-primary
ideal $\fa$ as well as a fixed set of generators
$f_1,\dots,f_m$ of $\fa$. 

Consider a multi-index 
$\cI=\{i_1,\dots,i_n\}\subseteq\{1,\dots,m\}$.
Let $\pi_\cI:\P^{m-1}\setminus W_\cI\to\P^{n-1}$,
where $W_\cI:=\{w_{i_1}=\dots=w_{i_n}=0\}\subset\P^{m-1}$,
be the projection given by 
$[w_1:\dots:w_m]\to[w_{i_1}:\dots:w_{i_n}]$.
Define $\Psi_\cI:X\dashrightarrow\P^{n-1}$ by
$\Psi_{\cI}:=\pi_{\cI}\circ\Psi$.
\begin{df}
  Let $E\subset X$ be an exceptional prime.
  We say that $\cI$ is \emph{$E$-essential} or that $\I$ is \emph{essential with respect to $E$} if 
  $\Psi(E)\not\subset W_\cI$ and 
  if $\Psi_{\cI}|_E:E\dashrightarrow\P^{n-1}$ is
  dominant, that is, $\Psi_\cI(E)$ is not contained in 
  a hypersurface. We say that $\I$ is \emph{essential} if it is essential with respect to at least one exceptional prime.
\end{df}

If $\cI$ is $E$-essential, then $E$
must be a Rees divisor of $\fa$, so, in fact, $\I$ is essential if it is essential with respect to at least one Rees divisor. 
Conversely, if $E$ is 
Rees divisor of 
$\fa$, then there exists at least one 
$E$-essential multi-index ~ $\cI$. Observe, however, that $\I$ can be essential with respect to more than one $E$, and conversely that there can be several $E$-essential multi-indices; compare to the discussion at the end of Section ~\ref{monomialcase} and the examples in Section ~\ref{teflon}.

Consider an exceptional prime $E$ of $\pi$ and 
a point $x\in E$ not lying on any other exceptional prime.
Pick local coordinates $(\sigma_1,\dots,\sigma_n)$ at $x$ such that
$E=\{\sigma_1=0\}$. We can write $f_i=\sigma_1^af'_i$, for 
$1\le i\le m$, where 
$a=\ord_E(\fa)$ and $f'_i\in\cO_{X,x}$. The holomorphic
functions $f'_i$ can be viewed as local sections of 
the line bundle $\cO_X(-Z)$ and there exists at least 
one $i$ such that $f'_i(x)\ne0$.
\begin{lma}\label{bralemma}
  A multi-index $\cI=\{i_1,\dots,i_n\}$ is $E$-essential if and only if 
the form 
\begin{equation}\label{rosa}
\sum_{k=1}^n (-1)^{k-1} f'_{i_k} df'_{i_1}\wedge\dots\wedge
    \widehat{df'_{i_k}}\wedge\dots\wedge df'_{i_n} 
\end{equation}
is generically nonvanishing on $E$. 
\end{lma}
\begin{remark}\label{regn}
Observe in particular that
\begin{equation}\label{regna} 
\ord_E(f_{i_1})=\ldots=\ord_E(f_{i_n})=\ord_E(\fa)
\end{equation}
 if $\I$ is $E$-essential.

\end{remark}

\begin{proof}
Locally on $E$ (where $f_j'\neq 0$) we have that 
\[
\Psi_\I=
\left [\frac{f_1'}{f_j'}: \ldots : \frac{f_{j-1}'}{f_j'}:\frac{f_{j+1}'}{f_j'}:
\ldots \frac{f_n'}{f_j'}
\right ].
\]
Note that $\Psi_\I$ is dominant if (and only if) $\jac (\Psi_\I)$ is generically nonvanishing, or equivalently the holomorphic form 
\begin{equation}\label{twin}
\partial \left ( \frac{f_1'}{f_j'} \right )\wedge
\ldots\wedge
\partial \left ( \frac{f_{j-1}'}{f_j'} \right )\wedge
\partial \left ( \frac{f_{j+1}'}{f_j'} \right )\wedge
\ldots\wedge
\partial \left ( \frac{f_n'}{f_j'} \right )
\end{equation}
is generically nonvanishing. But \eqref{twin} is just a nonvanishing function times \eqref{rosa}. 
\end{proof}
%
%

 \section{Proof of Theorem B}\label{coeff}
Throughout this section let $\a$ denote the ideal $(f)$. 
Let us first prove that $R^f_\I\not\equiv 0$ implies that $\I$ is essential.
Let $\pi: X\to (\C^n,0)$ be a log-resolution of $\a$. 
By standard arguments, see ~\cite{PTY}, ~\cite{A} etc., the analytic continuation to $\lambda=0$ of 
\begin{equation}\label{uppe}
\pi^*(\dbar|f|^{2\lambda}\wedge u)
\end{equation}
exists and defines a globally defined current on $X$, whose push-forward by $\pi$ is equal to $R$; we denote this current by $\widetilde R$, so that $R=\pi_*\widetilde R$. Indeed, provided that the analytic continuation of \eqref{uppe} exists, we get by the uniqueness of analytic continuation
\begin{multline}\label{muck}
\pi_* \widetilde R\cdot\Phi=
\pi_* (\pi^* (\dbar |f|^{2\lambda}\wedge u))\cdot\Phi|_{\lambda=0}=\\
\pi^* (\dbar |f|^{2\lambda}\wedge u)\cdot\pi^*\Phi|_{\lambda=0}=
\dbar |f|^{2\lambda}\wedge u\cdot \Phi|_{\lambda=0}= R\cdot\Phi.
\end{multline}

In the same way we define currents 
\begin{equation*}
\widetilde R_\I=\pi^*(\dbar|f|^{2\lambda} \wedge u_\I)|_{\lambda=0},
\end{equation*}
where 
\begin{equation*}
u_\I=\frac{s_\I\wedge (\dbar s_\I)^{n-1}}{|f|^{2n}}.
\end{equation*}
Let $E$ be an exceptional prime and let us fix a chart $\U$ in $X$ such that $\U\cap E \neq \emptyset$ and local coordinates $\sigma$ so that the pull-back of $f$ is of the form $\pi^* f=\mu f'$, where $\mu$ is a monomial, $\mu=\sigma_1^{a_1}\cdots \sigma_n^{a_n}$ and $f'$ is nonvanishing, and moreover $E=\{\sigma_1=0\}$, see Section ~\ref{logres}. Then $\pi^* s_\I=\overline \mu s'_\I$ for some nonvanishing section $s'_\I$ and $\pi^*|f|^2=|\mu|^2\nu$, where $\nu=|s'|^2$ is nonvanishing. 
Hence, using \eqref{homogen} 
\begin{equation*}
\widetilde R_\I=
\dbar(|\mu|^{2\lambda}\nu^\lambda)\frac {s'_\I\wedge (\dbar s'_\I)^{n-1}}{\mu^{n}\nu^n}\Big |_{\lambda=0}
\end{equation*}
which by \eqref{envar} is equal to
\[
\sum_{i=1}^n
\left[\frac{1}{\sigma_1^{na_1}\cdots
\sigma_{i-1}^{na_{i-1}}\sigma_{i+1}^{na_{i+1}} \cdots \sigma_n^{na_n}}\right ]\dbar\left[\frac{1}{\sigma_i^{na_i}}\right ]
\wedge\frac {s'_\I\wedge (\dbar s'_\I)^{n-1}}{\nu^n}.
\]
Thus $\widetilde R$ and $\widetilde R_\I$ are pseudomeromorphic in the sense of ~\cite{AW2} and so it makes sense to take restrictions of them to subvarieties of their support, see Section ~\ref{pseudo}. 

\begin{lma}\label{sing}
Let $E$ be an exceptional prime. The current $\widetilde R_\I\1_E$ vanishes unless $\I$ is essential with respect to $E$. Moreover $\widetilde R_\I\1_E$ only depends on the $f_k$ which satisfy that $\vE(f_k)=\vE(\a)$.
\end{lma}

\begin{proof}
Recall (from Section ~\ref{pseudo}) that $\dbar[1/\sigma_i^{a}]$ has the standard extension property with respect to $E=\{\sigma_i=0\}$. 
Thus
\begin{equation}\label{star}
\widetilde R_\I\1_E=
\left[\frac{1}{\sigma_2^{na_2}\cdots \sigma_n^{na_n}}\right ]\dbar\left[\frac{1}{\sigma_1^{na_1}}\right ]
\wedge\frac {s'_\I\wedge (\dbar s'_\I)^{n-1}}{\nu^n} \1_E.
\end{equation}
It follows that  $\widetilde R_\I\1_E$ vanishes unless
\begin{equation*}
s'_\I\wedge (\dbar s'_\I)^{n-1}\1_E \not \equiv 0,
\end{equation*}
which by Lemma ~\ref{bralemma} is equivalent to that $\I$ is $E$-essential. Indeed, note that the coefficient of $f'\wedge (\dbar f')^{n-1}$ is $(n-1)!$ times \eqref{rosa}.

For the second statement, recall that 
$\nu=|s'|^2=\sum |\pi^* \bar f_k/\bar\sigma_1^{a_1}|^2$. Note that 
$\pi^* \bar f_k/\bar\sigma_1^{a_1}\1_E=0$ if and only if $\pi^* \bar f_k/\bar\sigma_1^{a_1}$ is divisible by $\bar \sigma_1$, that is, $\vE(f_k)>\vE(\a)$. Hence 
$\widetilde R_\I\1_E$ only depends on the $f_k$ for which $\vE(f_k)=\vE(\a)$, compare to \eqref{star}.  
\end{proof}

\begin{remark}\label{sepremark}
In light of the above proof, $\widetilde R\1_E$ has the SEP with respect to $E$. This follows since $\widetilde R\1_{E}$ is of the form \eqref{star} and $\dbar[1/\sigma_1^a]$ has the SEP with respect to $E=\{\sigma_1=0\}$, see Section ~\ref{pseudo}. 
\end{remark}

Next, let us prove that $R^f_\I \not\equiv 0$ as soon as $\I$ is essential. 
In order to do this we will use arguments inspired by ~\cite{A2}. 
Throughout this section let $\widetilde M_\I$ denote the current 
$\widetilde R_\I \wedge \pi^* (df_\I/ (2 \pi i))^n/n!$ on $X$. Here $e^*_{i_1}\wedge\cdots\wedge e^*_{i_n}\wedge e_{i_n}\wedge\cdots\wedge e_{i_1}=e^*_\I\wedge e_\I$ should be interpreted as 1 so that in fact $\pi_*(\widetilde M_\I) = 
R^f_\I \wedge df_{i_n}\wedge\cdots\wedge df_{i_1}/(2\pi i)^n$.

\begin{lma}\label{mcurrent}
The $(n,n)$-current $\widetilde M_\I$ is a positive measure on $X$ whose support is precisely the union of exceptional primes $E$ for which $\I$ is $E$-essential.
\end{lma}

\begin{proof}
Note that Lemma ~\ref{sing} implies that the support of $\widetilde M_\I$ is contained in the union of exceptional primes for which $\I$ is $E$-essential. Let $E$ be such a divisor and let us fix a chart $\U$ and local coordinates $\sigma$ as in the proof of Lemma ~\ref{sing}. 
Then $\widetilde R_\I\1_E$ is given by \eqref{star}. 
We can always write $s'_\I \wedge (\dbar s'_\I)^{n-1}$ as
\[
s'_\I \wedge (\dbar s'_\I)^{n-1} =
(\bar\beta 
\widehat {d\bar \sigma_1} + d\bar\sigma_1 \wedge \bar \gamma)
\wedge e_\I,
\]
where $\widehat{d\bar \sigma_1}$ denotes $d\bar \sigma_2\wedge\cdots \wedge d\bar \sigma_n$, $\beta$ is a holomorphic function, and $\gamma$ is a holomorphic form. 
Moreover, since $\I$ is $E$-essential,  $s'_\I \wedge (\dbar s'_\I)^{n-1}|_E = \beta|_E \widehat {d\bar \sigma_1} \wedge e_\I$ is generically nonvanishing by Lemma ~\ref{bralemma} (in particular, $\beta|_E$ is generically nonvanishing). 

Moreover, with $\overline {e_j}$ interpreted as $e_j^*$, we have
\begin{multline*}
\pi^*(df_\I)^n= \pi^*(\partial \bar s_\I)^n = \partial (\bar s_\I \wedge (\partial \bar s_\I)^{n-1})= \\
\partial (\sigma_1^{na_1}\cdots\sigma_n^{na_n}
(\beta  
\widehat{d\sigma_1} 
+ d\sigma_1 \wedge \gamma)
)\wedge e_\I^* 
=\\
na_1 \sigma_1^{na_1-1} 
(\sigma_2^{na_2} \cdots\sigma_n^{na_n} \beta + 
\sigma_1\delta) 
d \sigma
\wedge e_\I^*, 
\end{multline*}
where $\delta$ is some holomorphic function, $d \sigma$ denotes $d\sigma_1\wedge\cdots \wedge d \sigma_n$, and $e_\I^*=e_{i_1}^*\wedge\cdots\wedge e_{i_n}^*$. 

Hence, using \eqref{mjau}, we get
\begin{multline}\label{kol}
\widetilde M_\I\1_E=
\widetilde R_\I\1_E \wedge \left ( \frac {\pi^*(df_\I)}{2 \pi i} \right )_n = 
\\
\frac{1}{n!}
\left[\frac{1}{\sigma_2^{na_2}\cdots \sigma_n^{na_n}}\right ]\dbar\left[\frac{1}{\sigma_1^{na_1}}\right ]
\wedge\frac 
{\beta ~ \widehat {d\bar \sigma_1}}
{|f'|^{2n}} 
\1_E 
\\
\wedge na_1 \sigma_1^{na_1-1}
[\sigma_2^{na_2} \cdots \sigma_n^{na_n} \beta + 
 \sigma_1\delta] 
d \sigma
\wedge e_\I^*\wedge e_\I 
= \\
\frac{na_1}{(2\pi i)^n}
\dbar\left[\frac{1}{\sigma_1}\right ]
\frac{|\beta|^2}{|f'|^{2n}} 
\widehat{d\bar \sigma_1}\wedge d \sigma \1_E.
\end{multline}
The right hand side of \eqref{kol} is just Lebesgue measure on $E$ times a smooth, positive, generically nonvanishing function. Hence $\widetilde M_\I$ is a positive current whose support is precisely the union of exceptional primes $E$ for which $\I$ is $E$-essential.
\end{proof}

\begin{remark}
It follows from the above proof that $\widetilde M \1_E$ is absolutely continuous with respect to Lebesgue measure on $E$.
\end{remark}

To conclude, the only if direction of Theorem ~B follows immediately from Lemma ~\ref{sing}. Lemma ~\ref{mcurrent} implies that $\pi_*(\widetilde M_\I)=R_\I \wedge df_{i_n}\wedge\cdots\wedge df_{i_1}/(2\pi i)^n=$ is a positive current with strictly positive mass if $\I$ is essential. In particular, $R^f_\I\not\equiv 0$, which proves the if direction of Theorem ~B. Hence Theorem ~B is proved.

\section{Annihilators}\label{annsection}
We are particularly interested in the annihilator ideal of $R^f$. Recall from Theorem B that $R^f_\I\not\equiv 0$ if and only if $\I$ is essential. Hence
\begin{equation}\label{irreducible}
\ann R^f=\bigcap_{\I\text{ essential}} \ann R^f_\I.
\end{equation}
In this section we prove Theorem A, which gives estimates of the size of $\ann R^f$. We also prove Theorem C, which gives an explicit description of $R^f$ in case $(f)$ is a complete intersection ideal. In fact, Theorems A and C are consequences of Theorem ~\ref{annihilatorn} and Proposition ~\ref{denandra} below.

\begin{thm}\label{annihilatorn}
Suppose that $f=(f_1,\ldots,f_m)$ generates an $\m$-primary ideal $\a\subset \O^n_0$. Let $R^f=(R^f_\I)$ be the corresponding Bochner-Martinelli residue current. Then $\ann R^f=\a$ if and only if $\a$ is a complete intersection ideal, that is, $\a$ is generated by $n$ germs of holomorphic functions. 

Moreover if $\a$ is a complete intersection ideal, then for
$\I=\{i_1,\ldots,i_n\}\subseteq\{1,\ldots,m\}$ 
\begin{equation}\label{jojojo}
R^f_\I=C_\I~ 
\dbar\left [\frac{1}{f_{i_1}}\right ]\wedge\cdots\wedge 
\dbar\left [\frac{1}{f_{i_n}}\right ],
\end{equation}
where $C_\I$ is a non-zero constant if $f_{i_1},\ldots, f_{i_n}$ generates $\a$ and zero otherwise. 
\end{thm}

For $\I=\{i_1,\ldots, i_n\}\subseteq \{1,\ldots, m\}$, let $f_\I$ denote the tuple $f_{i_1}, \ldots, f_{i_n}$, which we identify with the section $\sum_{i\in\I}f_{i}e_{i}^*$ of $V$. To prove (the first part of) Theorem ~\ref{annihilatorn} we will need two results. 

The first result is a simple consequence of Lemma ~\ref{mcurrent}. 
Given a tuple $g$ of holomorphic functions $g_1,\ldots , g_n\in\O_0^n$, let $\jac (g)$ denote the Jacobian determinant $\det |\frac{\partial g_i}{\partial z_j}|_{i,j}$.  
\begin{lma}\label{pink}
We have that  $\jac (f_\I)\in \ann R^f_\I$ if and only if $R^f_\I\equiv 0$.
\end{lma}

\begin{proof}
The if direction is obvious. Indeed if $R^f_\I\equiv 0$, then $\ann R^f_\I=\O_0^n$.

For the converse, suppose that $R^f_\I\not\equiv 0$. 
From the previous section 
we know that this implies that $R_\I^f\wedge df_{i_n}\wedge\cdots\wedge df_{i_1}\not\equiv 0$. However the coefficient of $df_{i_n}\wedge\cdots\wedge df_{i_1}$ is just $\pm \jac (f_\I)$ and so $\jac (f_\I)\notin \ann R^f_\I$. 
\end{proof}

The next result is Theorem ~1.1 and parts of the proof thereof in ~\cite{Hi}.
Recall that the socle $\soc(N)$ of a module $N$ over a local ring $(R,\m)$ consists of the elements in $N$ that are annihilated by $\m$, see for example ~\cite{BH}.

\begin{thm}\label{hickel}
Assume that $g_1,\ldots,g_n$ generate an ideal $\a\subset\O_0^n$. Then $\jac(g_1,\ldots,g_n)\in\a$ if and only if $\codim V(\a) < n$. 

Moreover, if $\codim V(\a)=n$, then the image of $\jac (g)$ under the natural surjection $\O_0^n \to \O_0^n/\a$ generates the socle of $\O_0^n/\a$.
\end{thm}

\begin{lma}\label{lambi}
If $R^f_\I\not\equiv 0$ and $\codim V(f_\I)=n$, then $\ann R^f_\I\subseteq (f_\I)$.
\end{lma}

\begin{proof}
We claim that it follows that every $\m$-primary ideal $J\subset \O_0^n$ that does not contain $\jac (f_\I)$ is contained in $(f_\I)$. Applying the claim to $\ann R^f_\I\not\ni \jac (f_\I)$ (if $R^f_\I\not\equiv 0$) proves the lemma.

The proof of the claim is an exercise in commutative algebra; however, we supply the details for the reader's convenience. 
Suppose that $J\subset \O_0^n$ is an $\m$-primary ideal such that $\jac (f_\I)\notin J$, but that there is a $g\in J$ such that $g\notin (f_\I)$. The latter condition means that $0\neq \tilde g \in \tilde J$, where $\tilde g$ and $\tilde J$ denote the images of $g$ and $J$, respectively, under the surjection $\O_0^n \to \O_0^n/(f_\I)$. Then, for some integer $\ell$, 
$\m^\ell \tilde g \neq 0$ but $\m^{\ell+1} \tilde g =0$ in $A:= \O_0^n/(f_\I)$; in other words $\m^\ell \tilde g$ is in the socle of $A$. 
According to Theorem ~\ref{hickel}, the socle of $A$ is generated by $\jac (f_\I)$ and so it follows that $\jac (f_\I)\in\tilde J$. This, however, contradicts the assumption made above and the claim is proved.
\end{proof}

\begin{proof}[Proof of Theorem ~\ref{annihilatorn}]
We first prove that $\ann R^f=\a$ implies that $\a$ is a complete intersection ideal. Let us therefore assume that $\ann R^f = \a$.

We claim that under this assumption, $\codim V(f_\I)=n$ as soon as $\I$ is essential. To show this, assume that there exists an essential multi-index $\I=\{i_1,\ldots, i_n\} \subseteq \{1, \ldots, m\}$ such that $\codim V(f_\I)<n$. Then by Theorem ~\ref{hickel} $\jac (f_\I)\in (f_\I)\subseteq \a$. 
However, by Lemma ~\ref{pink} $\jac (f_\I)\notin \ann R^f_\I$. Thus we have found an element that is in $\a$ but not in $\ann R^f$, which contradicts the assumption. This proves the claim.

Next, let us consider the inclusion
\begin{equation}\label{babel}
\bigcap_{\I \text{ essential}} (f_\I)\subseteq \a.
\end{equation}
Assume that the inclusion is strict. By the claim above $\codim V(f_\I)=n$ if $\I$ is essential and so by Lemma ~\ref{lambi} 
\[
\ann R^f = \bigcap_{\I \text{ essential}} \ann R^f_\I 
\subseteq \bigcap_{\I \text{ essential}} (f_\I) \varsubsetneq \a,
\]
which contradicts the assumption that $\ann R^f = \a$. 
Hence equality must hold in \eqref{babel}, which means that $\a$ is generated by $f_\I$, whenever $\I$ is essential. (Note that there must be at least one essential multi-index if $R^f\not\equiv 0$.) To conclude, we have proved that $\ann R^f=\a$ implies that $\a$ is a complete intersection ideal. 

\smallskip

It remains to prove that if $\a$ is a complete intersection ideal, then $R^f_\I$ is of the form \eqref{jojojo} if $f_\I$ generates $\a$ and zero otherwise. Indeed, if $R^f_\I$ is given by \eqref{jojojo}, then $\ann R^f_\I = (f_\I)=\a$ by the classical Duality Principle; see the Introduction. This means that $\ann R^f_\I$ is either $\a$ or (if $R^f_\I\equiv 0$)  $\O_0^n$ and so $\ann R^f = \bigcap \ann R^f_\I = \a$.

Assume that $\a$ is a complete intersection ideal. 
Then, by Nakayama's Lemma $\a$ is in fact generated by $n$ of the $f_i$, compare to the discussion just before Theorem C. Assume that $\a$ is generated by $f_1,\ldots, f_n$; then $f_\ell=\sum_{j=1}^n \varphi_j^\ell f_j$ for some holomorphic functions $\varphi_j^\ell$. (Note that $\varphi_j^\ell=\delta_{j,\ell}$ for $\ell \leq n$.)  

We will start by showing that $R^f_\I$, where $\I=\{1,\ldots,n\}$, is of the form \eqref{jojojo}. 
Recall from Section ~\ref{rescurr} that 
\begin{equation}\label{forstas}
R_\I=
\dbar|f|^{2\lambda} \wedge \frac{s_\I\wedge (\dbar s_\I)^{n-1}}{|f|^{2n}}\bigg |_{\lambda=0}.
\end{equation}
Let us now compare \eqref{forstas} with the current $R(f_\I)$, that is, the residue current associated with the section $f_\I$ of the sub-bundle $\widetilde V$ of $V$ generated by $e_1^*,\ldots,e_n^*$. Since $\codim V(f_\I)=n$, the current $R(f_\I)$ is independent of the choice of Hermitian metric on $\widetilde V$ according to ~\cite[Proposition~2.2]{A}. More precisely, 
\begin{equation*}
R(f_\I)=
\dbar|g|^{2\lambda} \wedge \frac{\tilde s_\I\wedge (\dbar \tilde s_\I)^{n-1}}{\|f_\I\|^{2n}}\bigg |_{\lambda=0}, 
\end{equation*}
where $\|\cdot \|$ is any Hermitian metric on $\widetilde V$, $\tilde s_\I$ is the dual section of $f_\I$ with respect to $\|\cdot\|$, and $g$ is any tuple of holomorphic functions that vanishes at $\{f_\I=0\}=\{0\}$; in particular, we can choose $g$ as $f$. 

Let $\Psi$ be the Hermitian matrix with entries $\psi_{i,j}=\sum_{\ell=1}^m \varphi_i^\ell \bar \varphi_j^\ell$. Then $\Psi$ is positive definite and so it defines a Hermitian metric on $\widetilde V$ by $\|\sum_{i=1}^n \xi_i e_i\|^2=\sum_{1\leq i,j \leq n}\psi_{i,j} \xi_i \bar \xi_j$.
Observe that $\|f_\I\|^2= |f_1|^2+\cdots + |f_m|^2$ and moreover that $\tilde s_\I=\sum_{1\leq i,j \leq n} 
\psi_{i,j} \bar f_j e_i$. A direct computation gives that 
$\tilde s_\I \wedge (\dbar\tilde s_\I)^{n-1}=\det(\Psi) s_\I \wedge (\dbar s_\I)^{n-1}$. It follows that $R(f_\I)= C R_\I$, where $C=\det(\Psi(0))\neq 0$. 
By ~\cite[Theorem~1.7]{A} 
$R(f_\I)=\dbar[1/f_1]\wedge\cdots\wedge\dbar [1/f_n]\wedge e_n\wedge\cdots\wedge e_1$, and so we have proved that $R^f_\I$ is of the form ~\eqref{jojojo}. 

Next, let $\L$ be any multi-index $\{\ell_1,\ldots,\ell_n\} \subseteq \{1,\ldots,m\}$. By arguments as above 
$s_\L \wedge (\dbar s_\L)^{n-1}=\det(\bar \Phi_\L) s_\I\wedge (\dbar s_\I)^{n-1}$, where $\Phi_\L$ is the matrix with entries $\varphi_j^{\ell_i}$. Hence $R_\L=C_\L R^f_\I e_{\ell_n}\wedge\cdots\wedge e_{\ell_1}$, where $C_\L=\det(\bar \Phi_\L(0)) $. Note that $C_\L$ is non-zero precisely when $f_1,\ldots,f_n$ can be expressed as holomorphic combinations of $f_{\ell_1},\ldots,f_{\ell_n}$, that is, when $f_{\ell_1},\ldots,f_{\ell_n}$ generate $\a$. 
Hence $R_\mathcal L$ is of the form \eqref{jojojo} if $f_\L$ generates $\a$ and zero otherwise, and we are done.
\end{proof}

\begin{prop}\label{denandra}
Suppose that $f=(f_1,\ldots,f_m)$ generates an $\m$-primary ideal $\a\subset \O_0^n$, where $n\geq 2$. Let $R^f$ be the corresponding Bochner-Martinelli residue current. Then the inclusion
\begin{equation*}
\overline{\a^n}\subseteq \ann R^f
\end{equation*}
is strict. 
\end{prop}

Observe that Proposition ~\ref{denandra} fails when $n=1$. Then, in fact, $\a=\ann R^f=\overline \a$. 

\begin{proof}
We show that $\ann R^f\setminus \overline{\a^n}$ is non-empty. Consider multi-indices $\mathcal J=\{j_1,\ldots,j_n\}, \mathcal L=\{\ell_1,\ldots,\ell_n\}\subseteq\{1,\ldots,m\}$. By arguments as in the proof of Lemma ~\ref{mcurrent} one shows that 
\begin{equation*}
df_{j_1}\wedge\cdots\wedge df_{j_n}\wedge R^f_{\mathcal L}=
\jac(f_{\mathcal J})dz_1\wedge\cdots\wedge dz_n \wedge R^f_{\mathcal L}
\end{equation*}
either vanishes or is equal to a constant times the Dirac measure at the origin. Thus $z_k\jac(f_{\mathcal J}) R^f_{\mathcal L}=0$ for all coordinate functions $z_k$. It follows that $\m\jac(f_\I)\subseteq\ann R^f$ for all multi-indices $\I=\{i_1,\ldots,i_n\}$. 

Next, suppose that $\I=\{i_1,\ldots,i_n\}$ is essential with respect to a Rees divisor $E$ of $\a$. Then a direct computation gives that 
$\vE(df_{i_1}\wedge\ldots\wedge df_{i_n})=n\vE(\a)$ and $\vE(dz_1\wedge\ldots\wedge dz_n)\geq \sum_{i=1}^n\vE(z_i)-1$. Note that $\vE(z_k)\geq 1$ for $1\leq k\leq n$. Since $df_{i_1}\wedge\cdots\wedge df_{i_n}=
\jac(f_{\mathcal I})dz_1\wedge\cdots\wedge dz_n$ it follows that
\begin{equation*}
\vE(z_k\jac(f_\I))\leq n~\vE(\a)-n+1=\vE(\overline{\a^n})-n+1
\end{equation*}
for $1\leq k\leq n$. 
Hence, if $n\geq 2$, there are elements, for example $z_k \jac(f_\I)$, in $\m \jac (f_\I)$ that are not in $\overline{\a^n}$. This concludes the proof.
\end{proof}

\begin{proof}[Proofs of Theorems A and C]
Theorem A is an immediate consequence of (the first part of) Theorem ~\ref{annihilatorn} and Proposition ~\ref{denandra}. 

Suppose that $(f)$ is a complete intersection ideal. Then by Theorem ~B and (the second part of) Theorem ~\ref{annihilatorn} we have
\begin{equation*}
\I \text{ essential } \Leftrightarrow R^f_\I \not\equiv 0 
\Leftrightarrow f_\I \text{ generates } (f).
\end{equation*}
Moreover Theorem ~\ref{annihilatorn} asserts that in this case $R^f_\I$ is of the form \eqref{formen}. 
\end{proof}

\smallskip

\begin{remark}\label{sockelsockel}
Let us conclude this section by a partial generalization of Theorem 3.1 in ~\cite{W}. Even though we cannot explicitly determine $\ann R^f$ we can still give a qualitative description of it in terms of the essential multi-indices. 

The current $R^f_\I$ is a Coleff-Herrera current in the sense of Bj{\"o}rk ~\cite{Bj}, which implies that $\ann R^f_\I$ is irreducible, meaning that it cannot be written as an intersection of two strictly bigger ideals. Thus \eqref{irreducible} yields an \emph{irreducible decomposition} of $\ann R^f$, that is, a representation of the ideal as a finite intersection of irreducible ideals, compare to ~\cite[Corollary~3.4]{W2}. An ideal $\a$ in a local ring $A$ always admits an irreducible decomposition and the number of components in a minimal such is unique; if $\a$ is $\m$-primary it is equal to the minimal number of generators of the socle of $A/\a$, see for example ~\cite{HRS}. In light of \eqref{irreducible} we see that the number of components in a minimal irreducible decomposition of $\ann R^f$ is bounded from above by the number of essential multi-indices. 

In fact Lemma ~\ref{mcurrent} gives us even more precise information: if $\I$ is essential then $\soc(\O_0^n/\ann R^f_\I)$ is generated by the image of $\jac(f_\I)$ under the natural surjection $\O_0^n\to \O_0^n/\ann R^f_\I$. It follows that $\soc(\O^n_0/\ann R^f)$ is generated by the images of $\{\jac(f_\I)\}_{\I \text{ essential}}$ under the natural surjection $\O_0^n\to\O_0^n/\ann R^f$.
\end{remark}

\section{A geometric decomposition}\label{decompsection}
In this section we will see that the current $R^f$ admits a natural decomposition with respect to the Rees valuations of $\a=(f_1,\ldots,f_m)$. 

Given a log-resolution $\pi: X\to (\C^n,0)$ of $\a$, recall from Section ~\ref{coeff} that the analytic continuation of \eqref{uppe} defines a $\Lambda^n V$-valued current $\widetilde R$ on $X$, such that $\pi_* \widetilde R= R$. Let $\widetilde R^f$ denote the corresponding vector-valued current, that is, the current with the coefficients of $\widetilde R$ as entries. From Lemma ~\ref{sing} and Remark ~\ref{sepremark} we know that $\widetilde R^f$ has support on and the SEP with respect to the Rees divisors associated with $\a$. Hence $\widetilde R^f$ can naturally be decomposed as $\sum_{E \text{ Rees divisor}} \widetilde R^f \1_E$. Given a Rees divisor $E$ in $X$, let us consider the current $R^E:=\pi_* (\widetilde R^f\1_E)$.

\begin{lma}\label{oberoende}
The current $R^E$ is independent of the log-resolution. 
\end{lma}

\begin{proof}
Throughout this proof, given a log-resolution $\pi:X\to (\C^n,0)$, let $\widetilde R_X$ denote the current $\widetilde R$ on $X$, that is, the value of \eqref{uppe} at $\lambda=0$, and let $E_X$ denote the divisor on $X$ associated with the Rees valuation $\vE$. 

Any two log-resolutions can be dominated by a third, see for example  ~\cite[Example~9.1.16]{L2}. 
To prove the lemma it is therefore enough to show that 
$\pi_*(\widetilde R_X\1_{E_X})=\pi_*\varpi_*(\widetilde R_Y\1_{E_Y})$ for log-resolutions
\begin{equation*}
Y \stackrel{\varpi}{\longrightarrow} X \stackrel{\pi}{\longrightarrow} (\C^n,0)
\end{equation*}
of $\a$. 

We will prove the slightly stronger statement that 
$\widetilde R_X\1_{E_X}=\varpi_*(\widetilde R_Y\1_{E_Y})$. 
Observe that $\widetilde R_X=\varpi_* \widetilde R_Y$; compare to \eqref{muck}.
Moreover note that $\varpi^{-1}(E_X)=E_Y \cup \bigcup E'$, where each $E'$ is a divisor such that $\varpi(E')$ is a proper subvariety of $E_X$ (whereas $\varpi(E_Y)=E_X$). Let $A_Y = E_Y\setminus \bigcup E'$ and $A_X=\varpi(A_Y)$. Then $A_X$ and $A_Y$ are Zariski-open sets in $E_X$ and $E_Y$, respectively, and $\varpi^{-1}(A_X)=A_Y$. 
By Remark ~\ref{sepremark} $\widetilde R$ has the SEP with respect to the exceptional divisors, and so, using \eqref{joy} we can now conclude that
\begin{equation*}
\widetilde R_X\1_{E_X}=
\widetilde R_X\1_{A_X}=
\varpi_*(\widetilde R_Y\1_{A_Y})=
\varpi_*(\widetilde R_Y\1_{E_Y}).
\end{equation*}
\end{proof}

\begin{prop}\label{reesdec}
Suppose that $f=(f_1,\ldots,f_m)$ generates an $\m$-primary ideal $\a\subset\O^n_0$. Let $R^f$ be the corresponding Bochner-Martinelli residue current. 
Then
\begin{equation}\label{lupp}
R^f=\sum R^E,
\end{equation}
where the sum is taken over Rees valuations $\vE$ of $\a$ and $R^E$ is defined as above. Moreover each summand $R^E$ is $\not\equiv 0$ and depends only on the $f_j$ for which $\vE(f_j)=\vE(\a)$.
\end{prop}

\begin{proof}
Assume that $E$ is a Rees divisor. By Section ~\ref{essential} there is at least one $E$-essential multi-index; let $\I$ be such a multi-index. Then, by (the proof of) Theorem ~B the current $\pi_* (\widetilde R_\I\1_E)\not\equiv 0$, which means that $R^E$ has at least one nonvanishing entry.

We also get that $\widetilde R^f$ has support on the union of the Rees divisors. Moreover, by Remark ~\ref{sepremark} $\widetilde R^f\1_E$ has the SEP with respect to $E$. 
Thus
\[
\widetilde R^f=\widetilde R^f\1_{\bigcup_{E \text{ Rees divisor}} E}=
\sum_{E \text{ Rees divisor}} \widetilde R^f\1_E,
\]
which proves \eqref{lupp}.

The last statement follows immediately from the second part of Lemma ~\ref{sing}.
\end{proof}

\section{The monomial case}\label{monomialcase}
Let $\a \subset \O_0^n$ be an $\m$-primary monomial ideal generated by monomials $z^{a^j}$, $1\leq j\leq m$. Recall that the Newton polyhedron $\np(\a)$ is defined as the convex hull in $\mathbb R^n$ of the exponent set $\{a^j\}$ of $\a$. The Rees-valuations of $\a$ are monomial and in 1-1 
correspondence with the compact facets (faces of maximal dimension) of $\np(\a)$. 
More precisely the facet $\tau$ with normal vector $\rho=(\rho_1, \ldots, \rho_n)$ corresponds to the monomial valuation
$\ord_\tau(z_1^{a_1}\cdots z_n^{a_n})=\rho_1a_1+\ldots + \rho_na_n$, see for example ~\cite[Theorem~10.3.5]{HS}.

Let us interpret our results in the monomial case. First, consider the notion of essential multi-indices. Note that a monomial $z^{a}\in\a$ satisfies that $\ord_\tau(z^{a})=\ord_\tau(\a)$ precisely if $a$ is contained in the facet $\tau$. Thus in light of \eqref{regna} a necessary condition for $\I=\{i_1,\ldots, i_n\}\subseteq\{1,\ldots,m\}$ to be $E_\tau$-essential (if $E_\tau$ denotes the Rees divisor associated with $\tau$) is that $\{a^i\}_{i\in\I}$ are all contained in $\tau$. Moreover, for \eqref{rosa} to be nonvanishing the determinant $|a^i|$ has to be non-zero; in other words $\{a^i\}_{i\in\I}$ needs to span $\R^n$. In ~\cite{W} an exponent set $\{a^i\}_{i\in\I}$ was said to be essential if all $a^i$ are contained in a facet of $\np(\a)$ and $|a^i|\neq 0$. Our notion of essential is thus a direct generalization of the one in ~\cite{W}. Moreover Theorem ~B can be seen as a generalization of (the first part of) Theorem 3.1 in ~\cite{W}, which asserts that $R^f_\I\not\equiv 0$ precisely if $\I$ is essential. In fact, Theorem 3.1 also gives an explicit description of $\ann R^f_\I$. Moreover, Theorem ~\ref{annihilatorn} and Proposition ~\ref{denandra} are direct generalizations of Theorem 3.2 and Corollary 3.9, respectively, in ~\cite{W}.

Concerning the decomposition in Section ~\ref{decompsection} observe that in the monomial case each multi-index $\I$ can be essential with respect to at most one Rees divisor. Indeed, clearly a set of points in $\R^n$ cannot be contained in two different facets and at the same time span $\R^n$. Hence in the monomial case the decomposition $R^f=(R^f_\I)$ 
is a refinement of the decomposition \eqref{lupp}; in fact the nonvanishing entries of $R^E$ are precisely the $R^f_\I$ for which $\I$ is $E$-essential. In particular, 
\begin{equation*}
\ann R =\bigcap \ann R^E \quad \text { and } \quad \ann R^E=\bigcap_{\I ~~ E-\text{essential}}\ann R^f_\I.
\end{equation*}
This is however not true in general. For example, if $n=m$, the set $\I=\{1,\ldots, n\}$ is essential with respect to all Rees divisors of $\a$ (and the number of Rees divisors can be $ > 1$). Also, in general, $\bigcap \ann R^E$ is strictly included in $\ann R$, see Example ~\ref{telifon}.

\section{Examples}\label{teflon}
Let us consider some examples that illustrate the results in the paper. 

\begin{ex}\label{concrete} ~\cite[Example~3.4]{W}
Let $\a\subset \O_0^2$ be the monomial ideal 
$(f_1,\ldots, f_5)=(z^8, z^6w^2, z^2w^3, zw^5, w^6)$. 
The exponent set of $\a$ is depicted in Figure 1, where we have also drawn $\np(\a)$. 
\begin{figure}\label{lefigur}
\begin{center}
\includegraphics{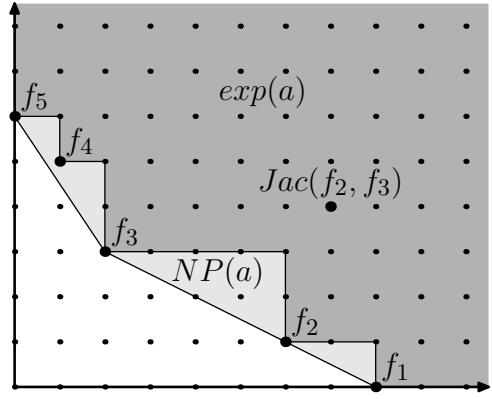}
\caption{The exponent set and Newton polyhedron of $\a$ in Example ~\ref{concrete}}
\end{center}
\end{figure}
The Newton polyhedron has two facets with normal directions $(1,2)$ and $(3,2)$ respectively. Thus there are two Rees divisors $E_1$ and $E_2$ associated with $\a$ with monomial valuations  
$\ord_{E_1}(z^a w^b)=a + 2 b$ and
$\ord_{E_2}(z^a w^b)=3 a + 2 b$, respectively. 
Now the index sets $\{1,2\}$, $\{1,3\}$, and $\{2,3\}$ are essential with respect to $E_1$ whereas $\{3,5\}$ is $E_2$-essential. Thus according to Theorem ~B $R^f$, which a priori has one entry for each multi-index $\{i,j\}\subseteq\{1,\ldots,5\}$, has four non-zero entries corresponding to the four essential index sets. 
Moreover, by Lemma ~\ref{pink} and Remark ~\ref{sockelsockel}, we have that for these index sets $\jac(f_\I)\notin \ann R^f$, whereas $\m\jac(f_\I)\subseteq\ann R^f$. For example, 
$\jac (z^6 w^2, z^2 w^3) = 14 z^7 w^4 \notin \ann R^f$, and thus, since $z^7 w^4 \in \a$, one sees directly that $\ann R^f\varsubsetneq \a$. Moreover $z\jac (z^6 w^2, z^2 w^3) = 14 z^8 w^4 \in \overline{\a^2}\setminus\ann R^f$.
\end{ex}

\begin{ex}\label{icke-monom}
Let $\a\subset \O_0^2$ be the product of the ideals $\a_1=(z, w^2)$, $\a_2=(z-w,w^2)$, and $\a_3=(z+w, w^2)$, each of which is
monomial in suitable local coordinates.
The ideal $\a_i$ has a unique (monomial) Rees-valuation $\ord_{E_i}$,
given by 
$\ord_{E_1}(z^a w^b) = 2a + b$, 
$\ord_{E_2}((z-w)^a w^b) = 2a + b$, and 
$\ord_{E_3}((z+w)^a w^b) = 2a + b$, respectively. 
By Corollary ~\ref{dimension2} the Rees-valuations of $\a$ are precisely $\ord_{E_1}$, $\ord_{E_2}$, and $\ord_{E_3}$.

Note that after blowing up the origin once,
the strict transform of $\a$ has support at exactly 
three points $x_1$, $x_2$, $x_3$ 
on the exceptional divisor; it follows that $\a$ is not a monomial ideal. 
A log-resolution $\pi:X\to(\C^2,0)$ 
of $\fa$ is obtained by further blowing up $x_1$, $x_2$ and $x_3$,
thus creating exceptional primes $E_1$, $E_2$ and $E_3$.

Now $\a$ is generated by 
\[
\{f_1,\ldots,f_4\}=
\{z(z-w)(z+w), ~z(z-w)w^2, ~z(z+w)w^2, ~(z-w)(z+w)w^2\}.
\]
Observe that none of these generators can be omitted; hence $\a$ is not a complete intersection ideal. Also, note that for each Rees divisor there is exactly one essential $\I\subseteq\{1,\ldots,4\}$. For example $\ord_{E_1}(f_1)=\ord_{E_1}(f_4)=\ord_{E_1}(\a)=4$, whereas $\ord_{E_1}(f_k)>4$ for $k=2,3$, and so $\I=\{1,4\}$ is the only $E_1$-essential index set. For symmetry reasons, $\{1,3\}$ is $E_2$-essential and $\{1,2\}$ is $E_3$-essential. 

Let us compute $R^f_{\{1,4\}}$. 
To do this, let $y\in X$ be the intersection point
of $E_1$ and the strict transform of $\{z=0\}$.
We choose coordinates $(\sigma,\tau)$ at $y$ so that
$E_1=\{\sigma=0\}$ and 
$(z,w)=\pi(\sigma,\tau)=(\sigma^2\tau,\sigma)$. Then  
$\pi^* s_{\{1,4\}}= \bar \sigma^4 (1-\bar\sigma^2\bar\tau^2) ( \bar \tau  e_1 + e_4)$
and it follows that
\[
\widetilde R_{\{1,4\}}=
- ~\dbar\left [\frac{1}{\sigma^8}\right ] \wedge \frac{d\bar \tau}{(1+|\tau|^2)^2}\wedge e_4\wedge e_1.
\]
Let $\phi=\varphi dw\wedge dz$ be a test form at $0\in\C^n$. 
Near $y\in X$ we have 
$\pi^* dw\wedge dz = \sigma^2 d\sigma \wedge d\tau$ and so 
\begin{multline*}
R^f_{\{1,4\}}\cdot \phi=  
\int 
\dbar\left [\frac{1}{\sigma^6}\right ] \wedge d\sigma \wedge \frac{d\bar \tau\wedge d\tau }{(1+|\tau|^2)^2}
~\varphi(\sigma^2 \tau, \sigma)
=\\
\frac{2\pi i}{5!} ~\varphi_{0,5}(0, 0) \int_\tau 
 \frac{d\bar \tau\wedge d\tau}{(1+|\tau|^2)^2}
= 
\frac{(2\pi i)^2}{5!} ~\varphi_{0,5}(0, 0)
= 
\dbar\left [\frac{1}{z}\right ]\wedge
\dbar\left [\frac{1}{w^6}\right ]\cdot\phi.
\end{multline*}
Hence 
$\ann R^f_{\{1,4\}}=(z,w^6)$. 
Similarly, 
$\ann R^f_{\{1,3\}}=(z-w,w^6)$ and 
$\ann R^f_{\{1,2\}}=(z+w,w^6)$, and so 
\[
\ann R^f = 
(z(z-w)(z+w), w^6).
\]
Note in particular that $\ann R^f\varsubsetneq \a$ in accordance with Theorem ~\ref{annihilatorn}. 
\end{ex}

\begin{ex}\label{fragan}
Let $\a\in\O^2_0$ be the monomial ideal $(z^2, zw, w^2)$ and let $f=f(B)$ be the tuple of generators: $f=(f_1,f_2,f_3)=(z^2, zw + w^2, B w^2)$. A computation similar to the one in Example ~\ref{icke-monom} yields that 
\begin{equation*}
R^f_{\{1,2\}} = 
C_0 ~\dbar\left [\frac{1}{z^3}\right ]\wedge \dbar\left [\frac{1}{w}\right ]
+ 2 ~ C_1 ~\dbar\left [\frac{1}{z^2}\right ]\wedge \dbar\left [\frac{1}{w^2}\right ],
\end{equation*}
where
\[
C_\ell=\frac{1}{2\pi i}\int\frac{|\tau|^{2\ell} d\bar\tau\wedge d\tau}
{(1+|\tau|^2|1+\tau|^2 + |B|^2|\tau|^4)^2}.
\]
Note that $R^f_{\{1,2\}}$ and its annihilator ideal depend not only on $f_1$ and $f_2$ but also on $f_3$. Indeed, a polynomial of the form $D z^2-Ew$ is in $\ann R^f_{\{1,2\}}$ if and only if $D/E=2C_1/C_0$, but $2C_1/C_0$ depends on the parameter $B$.  

However, $\ann R^f$ is independent of $B$. In fact, $\ann R^f_{\{1,3\}}=(z^2, w^2)$ and $\ann R^f_{\{2,3\}}=(z, w^3)$, which implies that $\ann R^f=\bigcap \ann R^f_\I=(z^3, z^2w, zw^2, w^3)$. 
\end{ex}

\begin{remark}\label{Dremark}
Example ~\ref{fragan} shows that the vector valued current $R^f$ depends on the choice of the generators of the ideal $(f)$ in an essential way. Still, in this example $\ann R^f$ stays the same when we vary $f$ by the parameter $B$. Also, we would get the same annihilator ideal if we chose $f$ as $(z^2, zw, w^2)$, see ~\cite[Theorem~3.1]{W}. 

We have computed several other examples of currents $R^f$ 
in all of which $\ann R^f$ is unaffected by a change of $f$ as long as the ideal $(f)$ stays the same. To be able to answer Question D in general, however, one probably has to understand the delicate interplay between contributions to $R^f$ and $R^f_\I$ from different Rees divisors, compare to Example ~\ref{telifon} below.
\end{remark}

\begin{ex}\label{telifon}
Let $\a\in\O^2_{0}$ be the complete intersection ideal $(f_1,f_2)=(z^3,w^2-z^2)$. After blowing up the origin the strict transform of $\a$ has support at two points $x_1$ and $x_2$ corresponding to where the strict transforms of the lines $z=w$ and $z=-w$, respectively, meet the exceptional divisor. Further blowing up these points yields a log-resolution of $\a$ with Rees divisors $E_1$ and $E_2$ corresponding to $x_1$ and $x_2$, respectively. 

A computation as in Example ~\ref{icke-monom} yields that
\begin{multline*}
2 R^{E_1}=
-\dbar\left [\frac{1}{z^4}\right ]\wedge \dbar\left [\frac{1}{w}\right ]
+\dbar\left [\frac{1}{z^3}\right ]\wedge \dbar\left [\frac{1}{w^2}\right ]\\
-\dbar\left [\frac{1}{z^2}\right ]\wedge \dbar\left [\frac{1}{w^3}\right ]
+\dbar\left [\frac{1}{z}\right ]\wedge \dbar\left [\frac{1}{w^4}\right ];
\end{multline*}
$R^{E_2}$ looks the same but with the minus signs changed to plus signs. 
Hence
\begin{equation*}
R^f=R^{E_1}+R^{E_2}=
\dbar\left [\frac{1}{z^3}\right ]\wedge \dbar\left [\frac{1}{w^2}\right ]
+\dbar\left [\frac{1}{z}\right ]\wedge \dbar\left [\frac{1}{w^4}\right ] . 
\end{equation*}
Note that $\ann R^f$ is indeed equal to $\a$, which we already knew by the Duality Principle. 
Observe furthermore that $z^3 R^{E_1}= - \dbar [1/z]\wedge\dbar [1/w]$, so that $z^3\notin\ann R^{E_1}$. Hence we conclude that in general
\begin{equation*}
\bigcap \ann R^E\varsubsetneq\ann R^f.
\end{equation*} 
\end{ex}

\def\listing#1#2#3{{\sc #1}:\ {\it #2},\ #3.}


\begin{thebibliography}{9999}

\bibitem{A}\listing{M.\ Andersson}
{Residue currents and ideals of holomorphic functions.}
{Bull.\ Sci.\ Math. {\bf 128} (2004), no. 6 481--512}

\bibitem{A2}\listing{M.\ Andersson}
{Residues of holomorphic sections and Lelong currents}  
{Ark.\ Mat.\  {\bf 43}  (2005),  no. 2, 201--219}


\bibitem{A6}\listing{M.\ Andersson}
{Uniqueness and factorization of Coleff-Herrera currents}{Preprint, to appear in Ann.\ Fac.\ Sci.\ Toulouse Math}

\bibitem{AG}\listing{M.\ Andersson \& E.\ G{\"o}tmark}
{Explicit representation of membership of polynomial ideals} 
{Preprint, G{\"o}teborg, available at arXiv:0806.2592}

\bibitem{ASS}\listing{M.\ Andersson \& H.\ Samuelsson \& J.\ Sznajdman}
{On the Briancon-Skoda theorem on a singular variety}
{Preprint, Göteborg, available at arXiv:0806.3700}

\bibitem{AW}\listing{M.\ Andersson \& E.\ Wulcan}
{Residue currents with prescribed annihilator ideals}
{Ann.\ Sci.\ École Norm.\ Sup.\  \textbf{40} (2007),  no. 6, 985--1007}

\bibitem{AW2}\listing{M.\ Andersson \& E.\ Wulcan}
{Decomposition of residue currents}
{to appear in Journal f{\"u}r die reine und angewandte Mathematik, available at arXiv:0710.2016}

\bibitem{BY}\listing{C.\ A.\ Berenstein \& A.\ Yger} 
{Analytic residue theory in the non-complete intersection case}  
{J.\ Reine Angew.\ Math.\  {\bf 527}  (2000), 203--235}

\bibitem{BGVY}\listing{C.\ A.\ Berenstein \& R.\ Gay \& A.\ Vidras \&
A.\ Yger} {Residue currents and Bezout identities}
{Progress in Mathematics
{\bf 114} Birkh\"auser Verlag (1993)}


\bibitem{Bj}\listing{J-E.\ Bj\"ork}{Residues and $\mathcal D$-modules} 
{The legacy of Niels Henrik Abel,  605--651, Springer, Berlin, 2004}

\bibitem{BS}\listing{J.\ Brian\c{c}on, H.\ Skoda }
{Sur la cl\^oture int\'egrale d'un id\'eal de germes de fonctions holomorphes en un point de $\mathbb C^n$}
{C.\ R.\ Acad.\ Sci.\ Paris S\'er.\ A {\bf 278} (1974) 949--951}

\bibitem {BH}\listing{W.\ Bruns \& J.\ Herzog}
{Cohen-Macaulay rings} 
{Cambridge Studies in Advanced Mathematics, {\bf 39}
Cambridge University Press, Cambridge (1993)}

\bibitem{CH}\listing {N.\ Coleff \& M.\ Herrera}
{Les courants résiduels associcés à une forme méromorphe}
{Lecture Notes in Mathematics \textbf{633} Springer Verlag, Berlin, 1978}

\bibitem{DS}\listing{A.\ Dickenstein  \& C.\ Sessa}
{Canonical representatives in moderate cohomology}
{Invent. Math. {\bf 80} (1985),  417--434}

\bibitem{Ha}\listing{R.\ Hartshorne}
{Algebraic Geometry}
{Graduate Texts in Mathematics, 52}
{Springer, New York, 1977}

\bibitem{HRS}\listing{W.\ Heinzer \& L.\ J.\ Ratliff \& K.\ Shah}
{On the irreducible components of an ideal}
{Comm.\ Algebra {\bf 25} (1997), no. 5, 1609--1634}

\bibitem{Hi}\listing{M.\ Hickel}
{Une note \`{a} propos du Jacobien de $n$ fonctions holomorphes  \`{a} l'origine de $\mathbb C^n$}{Preprint 2007}

\bibitem{HS}\listing{C.\ Huneke \& I.\ Swanson}
{Integral closure of ideals, rings, and modules, London Mathematical Society Lecture Note Series, 336} 
{Cambridge University Press, Cambridge, 2006} 

\bibitem{L2}\listing{R.\ Lazarsfeld} 
{Positivity in algebraic geometry. I \& II, Ergebnisse der Mathematik und ihrer Grenzgebiete. 3. Folge. Volumes 48 \& 49}
{Springer-Verlag, Berlin, {\bf 2004}}


\bibitem{P}\listing{M.\ Passare}
{Residues, currents, and their relation to ideals of holomorphic functions} 
{Math.\ Scand.\ {\bf 62} (1988), no. 1, 75--152}

\bibitem{PTY}\listing{M.\ Passare \& A.\ Tsikh \&  A.\ Yger}
{Residue currents of the Bochner-Martinelli type}
{Publ.\ Mat.  {\bf 44} (2000), 85--117}

\bibitem{T}\listing{B.\ Teissier}{Vari\'et\'es polaires. II. Multiplicit\'es polaires, sections planes, et conditions de Whitney \emph{Algebraic geometry (La R\'abida, 1981)}} {Lecture Notes in Mathematics {\bf 961} Springer Verlag, Berlin, 1982, pp 314--491}

\bibitem{VY}\listing{A.\ Vidras \& A.\ Yger}
{On some generalizations of Jacobi's residue formula}  
{Ann.\ Sci.\ École Norm.\ Sup.\ {\bf 34}  (2001),  no. 1, 131--157}



\bibitem{W}\listing{E.\ Wulcan}
{Residue currents of monomial ideals} 
{Indiana Univ.\ Math.\ J.\  {\bf 56}  (2007),  no. 1, 365--388}

\bibitem{W2}\listing{E. Wulcan}
{Residue currents constructed from resolutions of monomial ideals} 
{To appear in Math.\ Z.\, available at arXiv:math/0702847}

\end{thebibliography}
\end{document}